\documentclass{amsart}
\usepackage{amsfonts,amssymb,amscd,amsmath,enumerate,verbatim,calc}
\usepackage[all]{xy}

\newcommand{\CM}{Cohen-Macaulay}

\newcommand{\wrt}{with respect to}

\newcommand{\n}{\mathfrak{n} }
\newcommand{\m}{\mathfrak{m} }

\newcommand{\q}{\mathfrak{q} }

\newcommand{\ZZ}{\mathbb{Z} }

\newcommand{\R}{\mathcal{R}}
\newcommand{\M}{\mathcal{M}}

\newcommand{\rt}{\rightarrow}
\newcommand{\xar}{\longrightarrow}
\newcommand{\ov}{\overline}

\newcommand{\wt}{\widetilde }

\newcommand{\image}{\operatorname{image}}

\newcommand{\Ass}{\operatorname{Ass}}

\newcommand{\CMa}{\operatorname{CM}}

\newcommand{\depth}{\operatorname{depth}}
\newcommand{\Syz}{\operatorname{Syz}}
\newcommand{\Om}{\Omega }

\newcommand{\height}{\operatorname{height}}
\newcommand{\coker}{\operatorname{coker}}

\newcommand{\modA}{\operatorname{mod}}

\newcommand{\injdim}{\operatorname{injdim}}
\newcommand{\Min}{\operatorname{Min}}
\newcommand{\Hom}{\operatorname{Hom}}
\newcommand{\Ext}{\operatorname{Ext}}
\newcommand{\Tor}{\operatorname{Tor}}

\theoremstyle{plain}

\newtheorem{theorem}{Theorem}[section]
\newtheorem{corollary}[theorem]{Corollary}
\newtheorem{lemma}[theorem]{Lemma}
\newtheorem{proposition}[theorem]{Proposition}

\theoremstyle{definition}
\newtheorem{definition}[theorem]{Definition}
\newtheorem{s}[theorem]{}
\newtheorem{remark}[theorem]{Remark}

\theoremstyle{remark}

\begin{document}

\title[ associated graded modules]{A sub-functor for Ext and  Cohen-Macaulay associated graded modules with bounded multiplicity }
 \author{Tony J. Puthenpurakal}
\date{\today}
\address{Department of Mathematics, Indian Institute of Technology Bombay, Powai, Mumbai 400 076, India}
\email{tputhen@math.iitb.ac.in}

\begin{abstract}
Let $(A,\m)$ be a Henselian \CM \ local ring  and let \\ $\CMa(A)$ be the category of maximal \CM \ $A$-modules. 
We construct  $T \colon \CMa(A)\times \CMa(A) \rt \modA(A)$, a subfunctor of $\Ext^1_A(-, -)$ 
and use it to study  properties of associated graded modules over 
$G(A) = \bigoplus_{n\geq 0} \m^n/\m^{n+1}$, the associated graded ring of $A$.
As an application we give several examples of complete \CM \ local rings
$A$ with $G(A)$ \CM \ and having distinct indecomposable maximal \CM \ modules $M_n$ with $G(M_n)$ \CM  \ and 
the set $\{e(M_n)\}$ bounded (here
$e(M)$ denotes multiplicity of $M$).
\end{abstract}

\maketitle

\section{introduction}
Let $(A,\m)$ be a Henselian Noetherian local ring.  Recall that $A$ satisfies Krull-Schmidt property,
i.e., every finitely generated $A$-module is uniquely a direct sum of indecomposable $A$-modules.
Now assume that $A$ is \CM. Then we say $A$ is of finite  (\CM) representation type if $A$ has only finitely 
many indecomposable maximal \CM \ (MCM) $A$-modules.  Auslander proved that in this case $A$ is an isolated singularity, for
instance see \cite[Theorem 4.22]{Y}.
If in addition $A$ is equicharacteristic (containing a) perfect residue field then
Dieterich and Yoshino (independently) proved that if $A$ is an isolated singularity and not of finite representation type
then $A$ satisfies the first Brauer-Thrall conjecture (made for Artin algebra's), i.e., there exists indecomposable MCM $A$-modules $M_n$ with
$\{ e(M_n) \}_{n \geq 1} $ unbounded (here
$e(M)$ denotes multiplicity of $M$), see \cite[Theorem 6.2]{Y}. If $A$ is not an isolated singularity then
it follows from work of Huneke and Leuschke \cite[Theorem 1]{HL} that
$A$ has indecomposable MCM  $A$-modules $M_n$ such that $\{ e(M_n) \}_{n \geq 1} $ is bounded. 
We call this property as \emph{weak Brauer-Thrall II}.

Let $G(A) = \bigoplus_{n \geq 0} \m^n/\m^{n+1}$ be the associated graded ring of $A$ and 
if $M$ is a finitely generated $A$-module then let $G(M) = \bigoplus_{n \geq 0}\m^n M/\m^{n+1}M$ be the 
associated graded module of $M$. \emph{Note that we will only take associated graded modules \wrt \ $\m$}. Assume $G(A)$ is \CM. There are two natural questions that arise.
\begin{enumerate}
 \item Does there exist a \emph{non-free} MCM $A$-module $M$ with \CM \ associated graded module.
 \item How many indecomposable   MCM  $A$-modules exist with \CM \ associated graded modules. This naturally splits into two
 sub-questions:
\begin{enumerate}
 \item
 (\emph{Brauer-Thrall-I}). Does there exist indecomposable MCM modules $\{ M_n \}_{n \geq 1}$ with $G(M_n)$ \CM \ and $e(M_n) \rt + \infty$.
 \item
(\emph{weak Brauer-Thrall-II)}  Does there exist distinct indecomposable  \\ MCM modules $\{ M_n \}_{n \geq 1}$ with $G(M_n)$ \CM \ and
 $e(M_n)$ bounded. 
\end{enumerate}
\end{enumerate}

We now discuss what is previously known regarding these questions. \\
(1) This is  known for $\dim A \leq 1$. In one of the preliminary results in this paper we settle the $\dim A = 2$ case affirmatively.
It is also easy to see that if $A$ has minimal multiplicity then every MCM \ $A$-module has \CM \ associated 
graded module (for instance see \cite[Theorem 16]{Pu1}).
If $A$ is a strict complete intersection (i.e., $G(A)$ is also a complete intersection)
and $A$ is a quotient of a regular local ring then $A$ has an Ulrich module $U$, see
\cite[2.5]{HUB}. Recall an Ulrich module $U$ is an MCM $A$-module if its multiplicity equals its number of minimal generators. It
is well known that if $U$ is Ulrich then $G(U)$ is \CM. If $A = \widehat{R}$ where $R$ is a two dimensional \CM \ standard graded
algebra (and a domain) over an infinite field then 
also $A$ has an Ulrich module, see \cite[4.8]{BHU}. 

2(a) If $A$ is a \CM \ isolated singularity, convergent power series ring over a perfect field and having minimal multiplicity and not of finite representation
type then Brauer-Thrall-I holds for associated graded modules.
In a previous work,  the author proved that if $A$ is a complete  equi-characteristic hypersurface ring (and an isolated singularity)
with algebraically closed residue field and even dimension  (and of infinite representation type) then
there exists indecomposable  Ulrich  $A$-modules $\{ M_n \}$ with $\{ e(M_n) \}_{n \geq 1} $ unbounded, see \cite[1.11]{Pu-AR}. 

2(b) An easy case when this  holds is when $A$ has minimal multiplicty and is \emph{not} an isolated singularity.
To the best of the authors knowledge there is no other previous work discussing weak Brauer-Thrall II for associated graded modules.
The main goal of this paper is to give examples of \CM \ local rings satisfying weak Brauer-Thrall II 
i.e., for the existence of  distinct  MCM $A$-modules $M_n$ such that 
$G(M_n)$ is \CM \ and  $\{ e(M_n) \}_{n \geq 1} $ is a bounded set.

We show that the following classes of Henselain \CM \ local rings $A$ with $G(A)$ \CM \
satisfy weak Brauer-Thrall II  \\
(i)  $\dim A = 1, 2$ and $A$ is not an isolated singularity; see  Theorem \ref{BT-1-2}

(ii) Let $(Q,\n)$ be a Henselian regular local ring and let  $A = Q/(f_1,\ldots, f_c)$ be a strict complete intersection.
Let $f_1 = g^ih$ with $g$ irreducible, ($h$ is possibly a unit if $i \geq 2$ and is a non-unit if $i = 1$) and
$g$ does not divide $h$ . If $i \geq 2$ assume $\dim A \geq 1$.
If $i = 1$ assume $\dim A \geq 2$; see Theorem \ref{sci}

(iii) Let $(R,\n)$ be a \CM \ local ring having a non-free MCM module $E$ with $G(E)$ \CM.  Also assume $G(R)$ is \CM. Let $r \geq 1$ and let
$B = A[[X_1,\ldots, X_r]]$ or $B = R[X_1,\ldots, X_r]_{(\n, X_1,\ldots, X_r)}$. Note $G(B)$ is \CM.  Let $0 \leq l \leq r -1$ and let
$g_1,\ldots, g_l$ be such that $g_1^*, \ldots, g_l^*$ is $G(B)$ regular. Set $A = B/(g_1,\ldots, g_l)$; see Theorem \ref{rci}.

\emph{Construction used to prove our results:}\\
Let $(A,\m)$ be a \CM \ local ring and let $\CMa(A)$ be the category of maximal \CM \ $A$-modules. 
We construct  $T \colon \CMa(A)\times \CMa(A) \rt \modA(A)$, a sub-functor of $\Ext^1_A(-, -)$ 
as follows:

 Let $M$ be a MCM  $A$-module. Set
$$ e^T_A(M) = \lim_{n \rt \infty} \frac{(d-1)!}{n^{d-1}}\ell\left(\Tor^A_1(M, \frac{A}{\m^{n+1}}) \right ). $$
This function arose in the authors study of certain aspects of the  theory of Hilbert functions \cite{Pu1},\cite{Pu2}.
Using  \cite[Theorem 18]{Pu1} we get that $e^T_A(M)$ is a finite number  and it is zero if and only if $M$ is free.
Let $s \colon 0 \rt N \rt E \rt M \rt 0$ be an exact sequence
of MCM $A$-modules. Then by \cite[2.6]{Pu-L} we get that $e^T_A(E) \leq e^T_A(M) + e^T_A(N)$. Set $e^T(s) = e^T_A(M) + e^T_A(N) - e^T_A(E)$.
\begin{definition}
We say $s$ is $T$-split if $e^T_A(s) = 0$.
\end{definition}
\begin{remark}
Here $T$ stands for \emph{Tor}.
\end{remark}
\begin{definition}
 Let $M, N$ be MCM $A$-modules. Set 
 \[
  T_A(M,N) = \{ s \mid s  \ \text{is   a $T$-split extension} \}.
 \]
\end{definition}
We show
\begin{theorem}\label{funct-intro}(with notation as above)
  $T_A \colon \CMa(A)\times \CMa(A) \rt \modA(A)$ is  a sub-functor of $\Ext^1_A(-, -)$.
\end{theorem}
It is not clear from the definition whether $T_A(M,N)$ is non-zero. Our next results shows that there are plenty of $T$-split extensions
if $\dim \Ext^1_A(M,N) > 0$. We prove
\begin{theorem}\label{fl}
 Let $(A,\m)$ be a \CM \ local ring and let $M,N$ be MCM $A$-modules. Then
\[
 \Ext^1_A(M, N)/T_A(M, N) \quad \text{has finite length.}
\]
\end{theorem}

\s\label{T-G} If $s \colon 0 \rt N \rt E \rt M \rt 0$ is an extension 
of MCM $A$-modules then we have a \emph{complex} of $G(A)$-modules 
\[
 G(s) \colon 0 \rt G(N) \rt G(E) \rt G(M) \rt 0. 
\]
The utility of $T$-split sequences is
\begin{lemma}\label{T-CM}
 (with hypotheses as in \ref{T-G}) Assume $s$ is $T$-split and $G(N)$ is \CM. Then $G(s)$ is exact. In particular if
 $G(M)$ is also \CM \ then $G(E)$ is \CM.
\end{lemma}

Lemma \ref{T-CM} is used to prove the following  main technical result in our paper.
\begin{theorem}\label{main-sect-intro}
Let $(A,\m)$ be a  
Henselian \CM \ local ring of dimension $d \geq 1$. Suppose $M, N$ are MCM  modules with $G(M), G(N)$ \CM. If there exists only finitely many non-isomorphic
MCM $A$-modules $D$ with $G(D)$ \CM \ and $e(D) = e(M) + e(N)$; then 
$T_A(M, N)$ has finite length (in particular $\Ext^1_A(M, N)$ has finite length). If $h$  is the number of such 
isomorphism classes then $\m^h$
annihilates $T_A(M, N)$.
\end{theorem}
\begin{remark}
Notice statement of Theorem  \ref{main-sect}  is formally similar to the statement of a result by Huneke and  Leuschke \cite[Theorem 1]{HL}. 
The proof is similar too, except in few details which we describe in proof of this result in section 7.
\end{remark}
All our results follow by constructing suitable MCM modules $M, N$   with
$G(M)$, $G(N)$ \CM \ and $\dim \Ext^1_A(M, N) > 0$ and then appealing to Theorems \ref{fl} and \ref{main-sect-intro}. Our techniques also enable us to discuss weak Brauer-Thrall for Ulrich modules; see section \ref{sect-Ulrich}.

\begin{remark}
 Let $I$ be an $\m$-primary ideal and let $G_I(A) = \bigoplus_{n \geq 0}I^n/I^{n+1}$ be the associated graded ring of $A$ \wrt \ $I$.
 Now suppose $G_I(A)$ is \CM. Then we  can ask questions similar to the case when $I = \m$. However our technique fails in this case. See remark \ref{gen-m-prim} for an explanation.
\end{remark}

We now describe in brief the contents of this paper. In section two we discuss some preliminary results that we need. In section three we prove Theorem \ref{funct-intro}. In the next section we prove Theorem \ref{fl}. In  section
five we discuss a construction made in \cite{Pu5}. In the next section we prove Lemma \ref{T-CM}. In section seven we prove Theorem \ref{main-sect-intro} and two results analogous to it. In the next three section we give our examples showing existence of weak Brauer-Thrall-II for a large class of rings. 
\section{Preliminaries}
In this section we discuss a few preliminaries that we need.
Throughout all rings are commutative Noetherian and all modules considered are finitely generated unless otherwise stated.
The length of an $A$-module $M$ is denoted by $\ell(M)$ while $\mu(M)$ denotes the number of its minimal generators.

\s Let $(A,\m)$ be a local ring.
Let $N$ be a $A$-module of dimension $r$. It is well-known that there exists a polynomial $P_N(z) \in \mathbb{Q}[z]$
of degree $r$ such that $P_N(n) = \ell(N/\m^{n+1}N)$ for all $n \gg 0$.
We write
\[
P_N(z) = \sum_{i = 0}^{r}(-1)^ie_i(N)\binom{z+r-i}{r-i}.
\]
Then $e_0(N),\cdots,e_r(N)$ are integers and are called the \textit{Hilbert coefficients} of $N$. 
The number $e_0(N) = e(N) $ is called the \textit{multiplicity} of $N$. It is positive if $N$ is non-zero. The number $e_1(N)$ is \textbf{non-negative} if $N$ is \CM; see \cite[Proposition 12]{Pu1}.  Also note that
\[
\sum_{n\geq 0}\ell(N/\m^{n+1}N) z^n  = \frac{h_N(z)}{(1-z)^{r+1}};
\]
where $h_N(z)\in \ZZ[z]$ with $e_i(N) = h_N^{(i)}(1)/i!$ for $i = 0,\ldots,r$.

\s Let us recall the definition of superficial elements.  Let $N$ be an $A$-module. An element $x \in \m \setminus \m^2$ is said to be $N$-\textit{superficial} if there exists $c > 0$ such that $(\m^{n+1}N \colon x)\cap \m^cN = \m^nN$ for all $n \gg 0$. It is well-known that superficial elements exist when the residue field $k$ of $A$ is infinite.
If depth $N > 0$ then one can prove that a $N$-superficial element $x$ is $N$-regular. Furthermore  $(\m^{n+1}N \colon x) = \m^nN$ for all $n \gg 0$. 

Let $\dim N = r$. Then a sequence $x_1,\ldots, x_s \in \m$ (where $s \leq r$) is called an $N$-\textit{superficial sequence} if $x_i$ is  $N/(x_1,\ldots, x_{i-1})N$-superficial for all $i$. 

\s \emph{Minimal reduction:} For this notion we assume that the residue field of $A$ is infinite. Let $\dim N = r \geq 1$. We say  $J = (x_1, \ldots, x_r)$ is a minimal reduction of $N$ if $\m^{n+1}N  = J\m^n N$ for all $n \gg 0$.

Assume further that $N$ is \CM. Then it can be easily shown that if $x_1,\ldots, x_r$ is an $N$-superficial sequence then $J = (x_1,\ldots, x_r)$ is a minimal reduction of $N$.

\s \label{H-mod-sup} \emph{Behavior of Hilbert coefficients \wrt \ superficial elements:} Assume $N$ is an $A$-module with $\depth N > 0$ and dimension $r \geq 1$. Let $x$ be  $N$-superficial.
Then by \cite[Corollary 10]{Pu1} we have
\[
e_i(N/xN) = e_i(N) \quad \text{for} \ i = 0,\ldots, r-1.
\]

\s \label{sd} \emph{Sally Descent:} Assume $\depth M  \geq 2$ and $x$ is $M$-superficial. Set $N = M/xM$. If $\depth G(N) \geq 1$
then $\depth G(M) \geq 2$; see \cite[Theorem 8(2)]{Pu1}. 
\s Let $M$ be an $A$-module. We denote it's first syzygy-module by $\Om(M)$. If we have to specify the ring then we write it as $\Om_A(M)$. Recall
$\Om(M)$ is constructed  as follows: Let $G \xrightarrow{\phi} F \xrightarrow{\epsilon} M \rt 0$ be a minimal presentation of $M$. Then $\Om(M) = \ker \epsilon$. It is easily shown that  if $G^\prime \xrightarrow{\phi^\prime} F \xrightarrow{\epsilon^\prime} M \rt 0$ is another  minimal presentation of $M$ then
$\ker \epsilon \cong \ker \epsilon^\prime$. 

Set $\Om^1(M) = \Om(M)$. For $i \geq 2$  define $\Om^i(M) = \Om(\Om^{i-1}(M))$. It can be easily proved that $\Om^i(M)$ are invariant's of $M$.

\s \textbf{Base change:}
\label{AtoA'}
 Let $\phi \colon (A,\m) \rt (A',\m')$ be a flat local ring homomorphism with $\m A' = \m'$. If
 $N$ is an $A$-module set $N' = N\otimes A'$.
 In this case it can be seen that

\begin{enumerate}[\rm (1)]
\item
$\lambda_A(N) = \lambda_{A'}(N')$.
\item
 $\ell_A(M/\m^n M) = \ell_{A'}(M'/{\m'}^nM')$ for all $n \geq 0$.
\item
$\dim M = \dim M'$ and $\depth M = \depth M'$.
\item
 $e_i(M) = e_i(M')$ for all $i$.
\item
$\Om_A(M)\otimes_A A' \cong \Om_{A'}(M')$.
\end{enumerate}

 \noindent The specific base changes we do are the following:

(1) We can choose $A'$ to be the completion of $A$.

(2) If $k = A/\m$ is countable then we can choose
$A'$ with residue field uncountable. To do this note that by (1) we may assume $A$ is complete. Then set $A' = A[[X]]_{\m A[[X]]}$.
Note that the residue field of $A'$ is $k((X))$ which is uncountable.

\s \label{min-mult}
A \CM \ $A$-module $M$ is said to have \emph{minimal multiplicity} if  \\ $\deg h_M(z) \leq 1$. In this case 
it is known that $G(M)$ is \CM,  see \cite[Theorem 16]{Pu1}. 
If the ring $A$ has minimal multiplicity then it is easy to verify that every MCM  $A$-module has minimal multiplicity.

\s\label{Ulrich-exist}
Recall an $A$-module $U$ is said to be \emph{Ulrich} if $U$ is MCM and $e(M) = \mu(M)$. It is well-known that if $U$ is Ulrich
then $G(U)$ is \CM. Conversely it is known that if $M$ is an MCM $A$-module with $G(M)$ \CM \ then $M$ is Ulrich if and only if $e_1(M) = 0$.
Ulrich modules are special.
Let us first recall that the following cases
of CM rings when Ulrich modules are known to exist:
\begin{enumerate}
 \item $\dim A \leq 1$; (folklore; for readers convenience we prove it in \ref{Ulrich-d1}).
 \item $A$ is a strict complete intersection and a quotient of a regular local ring; see \cite[2.5]{HUB}
 \item $A$ is the completion (\wrt \ irrelevant maximal ideal) of a two dimensional standard graded \CM \ algebra (and a domain) over an infinite field; see \cite[4.8]{BHU}.
 
\item  If $A$ has minimal multiplicity then $\Syz_1(M)$ is Ulrich for any non-free MCM $A$-module (this is a folklore result and can be easily proved by reducing the 
exact sequence $0 \rt \Syz_1(M) \rt A^{\mu(M)} \rt M \rt 0$ by a maximal superficial sequence).
 \end{enumerate}

\section{$e^T(M)$}
In  this section  we prove Theorem \ref{funct-intro}.  This requires several preliminaries.

\s\label{basic} Let $(A.\m)$ be a \CM \ local ring. Let $M \in \CMa(A)$. In \cite[Prop. 17]{Pu1} we proved that the function
$$ n \mapsto \ell\left(\Tor^A_1(M, \frac{A}{\m^{n+1}}) \right )$$
is of polynomial type, i.e., it coincides with a polynomial $t_M(z)$ for all $n \gg 0$. In \cite[Theorem 18]{Pu1} we also proved that
\begin{enumerate}
\item
$M$ is free if and only if $\deg t_M(z) < d-1$.
\item
If $M$ is not free then $\deg t_M(z) = d-1$ and the normalized leading coefficient of $t_M(z)$ is 
$\mu(M)e_1(A) - e_1(M) - e_1(\Om(M))$; here $\mu(M)$ denotes the minimal number of generators of $M$.
\item
For any $M \in \CMa(A)$,
\begin{align*}
e^T_A(M) &= \lim_{n \rt \infty} \frac{(d-1)!}{n^{d-1}}\ell\left(\Tor^A_1(M, \frac{A}{\m^{n+1}}) \right ) \\
          &= \mu(M)e_1(A) - e_1(M) - e_1(\Om(M)).
\end{align*}
\end{enumerate}
\begin{remark}\label{gen-m-prim}
If $I$ is $\m$-primary we may consider the function
\[
n \mapsto \ell\left(\Tor^A_1(M, \frac{A}{I^{n+1}}) \right ).
\]
It is easy to prove that it is of polynomial type with degree $\leq d -1$. However  if $I \neq \m$ then  we have no control over its degree. It can be the zero polynomial even when $M$ is not free, see \cite[Remark 20]{Pu1}. Because of these reasons our technique fails for general $\m$-primary ideals.
\end{remark}
In a previous paper \cite[2.6]{Pu-L} we  proved that if $ 0 \rt N \rt E \rt M \rt 0$ is an exact sequence of
MCM $A$-modules then
\[
 e^T_A(E) \leq  e^T_A(M) + e^T_A(N).
\]

\s Let $\alpha \in \Ext_A^1(M,N)$. Let $\alpha$ be given by an extension $0 \rt N \rt E \rt M \rt 0$. Then note that $E$ is
a MCM $A$-module. Set
\[
 e^T_A(\alpha) = e^T_A(M) + e^T_A(N) - e^T_A(E).
\]
If $\alpha $ is given by an equivalent extension $0 \rt N \rt E^\prime \rt M \rt 0$ then $E \cong E^\prime$. Thus
$e^T_A(\alpha) $ is well-defined.
Note $e^T_A(\alpha) \geq 0$. If $\alpha = 0$ then the extension is split. So $e^T_A(\alpha) = 0$. 
\begin{definition}
 An extension $s\in \Ext^1_A(M,N)$ is $T$-split if $e^T_A(s) = 0$.
\end{definition}
Our next result shows that we can often reduce to dimension one.

In a previous work \cite[2.9]{Pu-L} we proved that $e^T_A(-)$ behaves well mod superficial elements.
\begin{proposition}\label{mod-sup}
Suppose $\dim A \geq 2$ and let $M \in \CMa(A)$. Assume the residue field $k$ is infinite. 
Let $x$ be $A \oplus M \oplus \Om_A(M)$-superficial. Set $B = A/(x)$ and $N = M/xM$. Then
\[
e^T_B(N) = e^T_A(M).
\]
\end{proposition}
As an easy consequence we get the following:
\begin{corollary}\label{mod-sup-2}
Suppose $\dim A \geq 2$ and let $M, N \in \CMa(A)$. Let $s \in \Ext^1_A(M,N)$ be represented by an exact sequence
$ 0 \rt N \rt E \rt M \rt 0$. Let $x$ be $A \oplus M \oplus N \oplus E \oplus \Om_A(M) \oplus \Om_A(N) \oplus \Om_A(E)$-superficial
element. Set $B = A/(x)$. Then
\[
e^T_A(s) = e^T_B(s\otimes B).
\]
Thus $s$ is $T$-split if and only if $s\otimes B$ is $T$-split.
\end{corollary}
In dimension one we have the following result:
\begin{lemma}\label{filmy}
 Let $(A,\m)$ be a \CM \ local ring of dimension one. Let $M, N$ be MCM  $A$-modules and let
 $s \colon  0 \rt N \rt E \rt M \rt 0$ be an extension of $M$ by $N$. The following assertions are equivalent:
 \begin{enumerate}[\rm (i)]
  \item 
  $s$ is $T$-split.
  \item
  For all $n \gg 0 $ the sequence 
  \[
   0 \rt \Tor^A_1(A/\m^{n+1}, N) \rt \Tor^A_1(A/\m^{n+1}, E) \rt \Tor^A_1(A/\m^{n+1}, M) \rt 0
  \]
is exact.
\item
For all $n \gg 0$ the map $\Tor^A_2(A/\m^{n+1}, E) \rt \Tor^A_2(A/\m^{n+1}, M) $ is surjective and the map 
$N/\m^{n+1} N \rt E/\m^{n+1}E$ is injective.
 \end{enumerate}
\end{lemma}
\begin{proof}
 As $\dim A = 1$ we get that for any MCM  $A$-module $X$ we have  $e^T_A(X) = \ell(\Tor^A_1(A/\m^{n+1}, X))$ for all $n \gg 0$.
 Thus (ii) $\implies$ (i). Furthermore it is clear that (ii) and (iii) are equivalent.
 
 We now assume (i), i.e., $s$ is $T$-split. It follows that $e^T_A(E) = e^T_A(M) + e^T_A(N)$. The exact sequence $0 \rt N \rt E
 \rt M \rt 0$ induces an exact sequence 
  \[
    \rt \Tor^A_1(A/\m^{n+1}, N) \xrightarrow{u_n} \Tor^A_1(A/\m^{n+1}, E) \xrightarrow{v_n} \Tor^A_1(A/\m^{n+1}, M) 
  \]
For $n\gg 0$ the module in the middle has the same length as the sum of the other two. It follows that for $n \gg 0$
the map $u_n$ is injective and the map $v_n$ is surjective. The result follows.
\end{proof}

\begin{proposition}\label{push-out}
 Let $(A,\m)$ be a \CM \ local ring of dimension $d \geq 1$ and let 
 $M, N, N^\prime, E, E^\prime$ be MCM $A$-modules. Suppose we have a commutative diagram
 \[
  \xymatrix
{
\alpha
\colon
 0
 \ar@{->}[r]
  & N
    \ar@{->}[d]
\ar@{->}[r]
 & E
    \ar@{->}[d]
\ar@{->}[r]
& M
    \ar@{->}[d]^{\xi}
\ar@{->}[r]
 &0
 \\
 \beta
 \colon
 0
 \ar@{->}[r]
  & N^\prime
\ar@{->}[r]
 & E^\prime
\ar@{->}[r]
& M
\ar@{->}[r]
&0
 }
\]
where $\xi$ is the identity map. 
If $\alpha $ is $T$-split then $\beta$ is also $T$-split.
\end{proposition}
\begin{proof}
 By \ref{mod-sup-2} we may assume that $\dim A = 1$. As $\alpha $ is $T$-split we have that for $i = 1,2$ the maps
 $u_{i,n} \colon \Tor^A_i(A/\m^{n+1}, E) \rt \Tor^A_i(A/\m^{n+1}, M)$ is surjective for $i = 1, 2$ and for all $n \gg 0$. We also have a 
 commutative diagram
 \[
    \xymatrix
{
 \Tor^A_i(A/\m^{n+1}, E)
    \ar@{->}[d]
\ar@{->}[r]^{u_{i,n}}
& \Tor^A_i(A/\m^{n+1}, M)
    \ar@{->}[d]^{\xi}
 \\
 \Tor^A_i(A/\m^{n+1},  E^\prime)
\ar@{->}[r]^{v_{i,n}}
& \Tor^A_i(A/\m^{n+1}, M)
 }
 \]
where $\xi$ is the identity map. A simple diagram chase shows that $v_{i,n}$ is also surjective for $i = 1, 2$ and for all $n \gg 0$.
\end{proof}

\begin{proposition}\label{pull-back}
 Let $(A,\m)$ be a \CM \ local ring of dimension $d \geq 1$ and let 
 $M, M^\prime, N, E, E^\prime$ be MCM $A$-modules. Suppose we have a commutative diagram
 \[
  \xymatrix
{
\alpha
\colon
 0
 \ar@{->}[r]
  & N
    \ar@{->}[d]^{\xi}
\ar@{->}[r]
 & E
    \ar@{->}[d]
\ar@{->}[r]
& M
    \ar@{->}[d]
\ar@{->}[r]
 &0
 \\
 \beta
 \colon
 0
 \ar@{->}[r]
  & N
\ar@{->}[r]
 & E^\prime
\ar@{->}[r]
& M^\prime
\ar@{->}[r]
&0
 }
\]
where $\xi$ is the identity map. 
If $\beta $ is $T$-split then $\alpha$ is also $T$-split.
\end{proposition}
\begin{proof}
 By \ref{mod-sup-2} we may assume that $\dim A = 1$. As $\beta $ is $T$-split we have that for $i = 0,1$ the maps
 $v_{i,n} \colon \Tor^A_i(A/\m^{n+1}, N) \rt \Tor^A_i(A/\m^{n+1}, E^\prime)$ is injective  for all $n \gg 0$. We also have a 
 commutative diagram
 \[
    \xymatrix
{
 \Tor^A_i(A/\m^{n+1}, N)
    \ar@{->}[d]^{\xi}
\ar@{->}[r]^{u_{i,n}}
& \Tor^A_i(A/\m^{n+1}, E)
    \ar@{->}[d]
 \\
 \Tor^A_i(A/\m^{n+1},  N)
\ar@{->}[r]^{v_{i,n}}
& \Tor^A_i(A/\m^{n+1}, E^\prime)
 }
 \]
where $\xi$ is the identity map. A simple diagram chase shows that
$u_{i,n}$ is also injective for $i = 0,1$ and for all $n \gg 0$.
The result follows.
\end{proof}

\begin{definition}
 Let $M, N$ be maximal \CM \ $A$-modules. Set 
 \[
  T_A(M, N) = \{ s \mid s\in \Ext^1_A(M,N) \ \text{and} \ e^T_A(s) = 0 \}.
 \]
\end{definition}
Our next result is 
\begin{theorem}\label{t-sub}
 Let $(A,\m)$ be a \CM \ local ring of dimension $d \geq 1$. Let $M, N$ be MCM $A$-modules. Then
 $T_A(M, N)$ is an $A$-submodule of $\Ext^A_1(M, N)$.
\end{theorem}
\begin{proof}
If $\alpha$ is the split extension then $e^T_A(\alpha) = 0$. So $0 \in T_A(M, N)$.
Let $\alpha \colon  0 \rt N \rt E \rt M \rt 0$ be an extension with $\alpha \in T_A(M, N)$. Then note that for $r \in A$ the extension $r\alpha$ is given by the 
push-out diagram:
 \[
  \xymatrix
{
\alpha
\colon
 0
 \ar@{->}[r]
  & N
    \ar@{->}[d]^{r}
\ar@{->}[r]
 & E
    \ar@{->}[d]
\ar@{->}[r]
& M
    \ar@{->}[d]^{\xi}
\ar@{->}[r]
 &0
 \\
 r \alpha
 \colon
 0
 \ar@{->}[r]
  & N
\ar@{->}[r]
 & E^\prime
\ar@{->}[r]
& M
\ar@{->}[r]
&0
 }
\]
where $\xi$ is the identity map. By Proposition \ref{push-out} we get that $r \alpha $ is also $T$-split.

Let $\alpha, \alpha^\prime \in T_A(M,N)$. Set $\alpha = 0 \rt N \rt E \xrightarrow{\pi} M \rt 0$ and $\alpha^\prime = 0 \rt N \rt E^\prime
\xrightarrow{\pi^\prime} M \rt 0$.
Consider the pull-back diagram:
\[
  \xymatrix
{
\beta
\colon
 0
 \ar@{->}[r]
  & N
    \ar@{->}[d]^\xi
\ar@{->}[r]
 & E^{\prime\prime}
    \ar@{->}[d]
\ar@{->}[r]
& E^\prime
    \ar@{->}[d]^{\pi^\prime}
\ar@{->}[r]
 &0
 \\
 \alpha
 \colon
 0
 \ar@{->}[r]
  & N
\ar@{->}[r]
 & E
\ar@{->}[r]^{\pi}
& M
\ar@{->}[r]
&0
 }
\]
where $\xi$ is the identity map. By Proposition \ref{pull-back} we get that $\beta$ is also $T$-split.
Let $\Delta = \{(n,-n) \in E^{\prime\prime} \mid n \in N \}$. Set $Y = E^{\prime\prime}/\Delta$.  It is well-known that $\alpha + \alpha^\prime $ is given by the extension:
\[
0 \rt N \rt Y \xrightarrow{\ov{\pi}} M \rt 0,
\]
where $\ov{\pi}([e,e']) = \pi(e)$ (here $[e,e^\prime]$ is the image of $(e,e')$ in $Y$).

We have the following commutative diagram

\[
  \xymatrix
{
\beta
\colon
 0
 \ar@{->}[r]
  & N
    \ar@{->}[d]^\xi
\ar@{->}[r]
 & E^{\prime\prime}
    \ar@{->}[d]^\delta
\ar@{->}[r]
& E^\prime
    \ar@{->}[d]^{\pi^\prime}
\ar@{->}[r]
 &0
 \\
 \alpha + \alpha^\prime
 \colon
 0
 \ar@{->}[r]
  & N
\ar@{->}[r]
 & Y
\ar@{->}[r]^{\ov{\pi}}
& M
\ar@{->}[r]
&0
 }
\]
where $\xi$ is the identity map and $\delta$ is the canonical surjection. We note that $\ker \delta = \ker \pi^\prime = N$.
From the middle column we get
\[
e^T(E^{\prime\prime}) \leq e^T(Y) + e^T(N).
\]
As $\alpha^\prime$ is $T$-split we get $e^T(N) = e^T(E^\prime) - e^T(M)$.
So we get 
\[
e^T(E^{\prime\prime}) \leq e^T(Y) + e^T(E^\prime) - e^T(M).
\]
As $\beta$ $T$-split we get we get $e^T(N) = e^T(E^\prime) - e^T(E^\prime)$.
It follows that
\[
e^T(N) \leq e^T(Y) - e^T(M).
\]
But the exact sequence for the extension $\alpha + \alpha^\prime$ yields 
\[
e^T(N) \geq e^T(Y) - e^T(M).
\]
It follows that $e^T(N) = e^T(Y) - e^T(M)$. So $\alpha + \alpha^\prime$ is $T$-split.
\end{proof}

We will need the following result in the next section. The first and third assertion are well-known.
\begin{lemma}\label{extend}
 Let $\phi \colon (A,\m) \rt (A',\m')$ be a flat local ring homomorphism of \CM \
 local rings
 with $\m A' = \m'$. If
 $C$ is an $A$-module set $C' = C\otimes A'$. Let $M, N$ be be MCM $A$-modules. Then
 \begin{enumerate}[\rm(1)]
 \item
 $M', N'$ are MCM  $A'$-modules.
 \item
 $e^T_A(M) = e^T_{A'}(M')$.
 \item
 If $\alpha \colon 0 \rt N \rt E \rt  M \rt 0$ is an extension then so is
 $$\alpha' = \alpha\otimes 1 \colon 0 \rt N' \rt E'  \rt M' \rt 0.$$
 \item
 $e^T_A(\alpha) = e^T_{A'}(\alpha')$.
 \item
 $\alpha \in T_A(M,N)$ if and only if $\alpha^\prime \in T_{A'}(M',N')$.
 \end{enumerate}
\end{lemma}
\begin{proof}
(1) This follows from \ref{AtoA'}(3)

(2) This follows from \ref{AtoA'}(4), (5).

(3) This is well-known.

(4) This follows from \ref{AtoA'}(4), (5) and (6).

(5) This follows from (4).
\end{proof}

We now state and prove the main result of this section.
\begin{theorem}\label{funct-main}(with notation as above)
  $T_A \colon \CMa(A)\times \CMa(A) \rt \modA(A)$ is  a sub-functor of $\Ext^1_A(-, -)$.
\end{theorem}
\begin{proof}
 Using \ref{t-sub} it suffices to prove $T_A$ is a functor.
 Let $\alpha \colon 0 \rt N \rt E \rt M \rt 0$ be an extension. Let $f \colon N \rt N^\prime$ be an $A$-linear map.
 Let $\beta \colon 0 \rt N^\prime \rt E^\prime \rt M \rt 0$ be the extension corresponding to 
 $\Ext^1_A(M, f)(\alpha)$. Then it is well-known (see \cite[p.\ 207]{Rot}) that we have a push-out diagram
 \[
  \xymatrix
{
\alpha
\colon
 0
 \ar@{->}[r]
  & N
    \ar@{->}[d]^{f}
\ar@{->}[r]
 & E
    \ar@{->}[d]
\ar@{->}[r]
& M
    \ar@{->}[d]^{\xi}
\ar@{->}[r]
 &0
 \\
 \beta
 \colon
 0
 \ar@{->}[r]
  & N^\prime
\ar@{->}[r]
 & E^\prime
\ar@{->}[r]
& M
\ar@{->}[r]
&0
 }
\]
where $\xi$ is the identity map. 
If $\alpha $ is $T$-split then by \ref{push-out} $\beta$ is also $T$-split.

Given a map $g\colon M^\prime \rt M$, let $\gamma \colon 0 \rt N \rt E^\prime \rt M^\prime \rt 0$
be the extension corresponding to $\Ext^1_A(g,N)(\alpha)$. Then it is well-known (see \cite[p.\ 208]{Rot}) that we have a
pull-back diagram

\[
  \xymatrix
{
\gamma
\colon
 0
 \ar@{->}[r]
  & N
    \ar@{->}[d]^{\xi}
\ar@{->}[r]
 & E^\prime
    \ar@{->}[d]
\ar@{->}[r]
& M^\prime
    \ar@{->}[d]^{g}
\ar@{->}[r]
 &0
 \\
 \alpha
 \colon
 0
 \ar@{->}[r]
  & N
\ar@{->}[r]
 & E
\ar@{->}[r]
& M
\ar@{->}[r]
&0
 }
\]
where $\xi$ is the identity map. 
By \ref{pull-back} it follows that    $\gamma$ is also $T$-split.
(note the change in labeling of modules here \wrt \  \ref{pull-back}).

Thus we have shown that $T_A\colon \CMa(A)\times \CMa(A) \rt \modA(A)$ is  a sub-functor 
of $\Ext^1_A(-, -)$.
\end{proof}
\section{Dimension of $T_A(M,N)$}
In this section we prove Theorem \ref{fl}. We re-state it for the reader's convenience. 
\begin{theorem}\label{dim-T}
Let $(A,\m)$ be \CM \ of dimension $d \geq 1$.
Let $M, N$ be MCM $A$-modules. Then
\[
\Ext^1_A(M,N)/T_A(M, N) \quad \text{has finite length.}
\]
\end{theorem}
The strategy of the proof is to prove first when $\dim A = 1$.
We will need the following preliminary results/notions.

We need the following result when $\dim A = 1$.
\begin{lemma}\label{dim1-lemm}
Let $(A,\m)$ be a \CM \ local ring of dimension one with infinite residue field $k = A/\m$. Let $M$ be a MCM $A$-module. Let $x \in \m$ be $A$-superficial. Say $\m^{c+1} = x \m^c$. Then we have the following:
\begin{enumerate}[\rm (1)]
\item
Fix $n \geq c$. The map $\psi \colon A/\m^n \rt A/\m^{n+1}$ given by 
$\psi(a + \m^n ) = xa + \m^{n+1}$ induces an isomorphism
$\Tor^A_1(M, A/\m^n) \cong \Tor^A_1(M, A/\m^{n+1}).$
\item
$e^T(M) = \ell \left(\Tor^A_1(M, A/\m^{n+1}) \right)$ for all $n \geq c $.
\end{enumerate}
\end{lemma}
\begin{proof}
(1) We have the following exact sequence
\[
0 \rt \frac{(\m^{n+1} \colon x)}{\m^n} \rt A/\m^n \xrightarrow{\psi}  A/\m^{n+1} \rt A/(\m^{n+1}, x) \rt 0.
\]
As $A$ is \CM \ we get that $x$ is $A$-regular. Furthermore as $\m^{c+1} = x \m^c$
 and $n \geq c$ we get that $(\m^{n+1} \colon x) = \m^n$.  Furthermore we have
 $\m^{n+1} \subseteq (x)$. So $A/(\m^{n+1}, x) = A/(x)$. Thus the above exact sequence reduces to 
 \[
 0 \rt  A/\m^n \xrightarrow{\psi}  A/\m^{n+1} \rt A/(x) \rt 0.
 \]
 
 As $M$ is an MCM $A$-module we get that $x$ is also $M$-regular.
Thus \\ $\Tor^A_i(M, A/(x)) = 0$ for $i \geq 1$. 
  Applying the functor $M \otimes -$ we get  the required result.
  
  (2) By definition we get that $e^T(M) = \ell(\Tor^A_1(M, A/\m^{n+1}))$ for all $n \gg 0$. By (1) we get the required result. 
\end{proof}
The following result is a basic ingredient to prove Theorem  \ref{dim-T} when $\dim A = 1$.
\begin{lemma}\label{dim1-lem-T}
Let $(A,\m)$ be a \CM \ local ring of dimension one with infinite residue field $k = A/\m$.  Let $x \in \m$ be $A$-superficial. Say $\m^{c+1} = x \m^c$. Let $M$, N be
MCM $A$-modules and let $\alpha \in \Ext^1_A(M,N)$. Let
$u \in \m^{n}$ where $n \geq c + 1$. Then
\begin{enumerate}[\rm (1)]
\item
$e^T(u\alpha) \leq e^T(\alpha)$.
\item
$e^T(u\alpha) = e^T(\alpha)$ if and only if $\alpha$ is $T$-split.
\end{enumerate}
 \end{lemma}

\s \label{MT-dim1} Before proving Lemma \ref{dim1-lem-T} let us prove Theorem \ref{dim-T} when $\dim A =1$ and $A$ contains an infinite residue field.
\begin{proof}
(with setup as in Lemma \ref{dim1-lem-T}). Let $\alpha \in \Ext^1_A(M,N)$. Let
$u \in \m^{c+1}$ be $A$-regular.  Then by Lemma \ref{dim1-lem-T}  for all $i \geq 1$ we have
\[
0 \leq \cdots \leq e^T(u^{i+1}\alpha) \leq e^T(u^{i}\alpha) \leq \cdots \leq e^T(u\alpha) \leq e^T(\alpha)
\]
It follows that  there exists $i_0$ such that $e^T(u^{i+1}\alpha) = e^T(u^{i}\alpha) $ for all $i \geq i_0$. It follows from Lemma \ref{dim1-lem-T} that
$u^i\alpha$ is $T$-split all $i \geq i_0$.

Now let $\alpha_1, \ldots, \alpha_s$ generate $\Ext^1_A(M,N)$ as an $A$-module.
By the previous argument we get that there exists $l$ such that 
$u^i\alpha_j$ is $T$-split  for all $i \geq l$   and for $j = 1,\ldots, s$.  So $u^l \Ext^1_A(M,N) \subseteq T_A(M,N)$. As $u^l$ is $A$-regular the result follows. 
\end{proof}
We now give
\begin{proof}[Proof of Lemma \ref{dim1-lem-T}]
Let $\alpha \colon  0 \rt N \rt C \rt M \rt 0$.  Then note that for $u \in \m^{n+1}$ the extension $u\alpha$ is given by the 
push-out diagram:
 \[
  \xymatrix
{
\alpha
\colon
 0
 \ar@{->}[r]
  & N
    \ar@{->}[d]^{u}
\ar@{->}[r]
 & C
    \ar@{->}[d]
\ar@{->}[r]
& M
    \ar@{->}[d]^{\xi}
\ar@{->}[r]
 &0
 \\
 u \alpha
 \colon
 0
 \ar@{->}[r]
  & N
\ar@{->}[r]
 & E
\ar@{->}[r]
& M
\ar@{->}[r]
&0
 }
\]
where $\xi$ is the identity map. It suffices to prove the following assertions
\begin{enumerate}[\rm (i)]
\item
$e^T(E) \geq e^T(C)$.
\item
$e^T(E) =  e^T(C)$ if and only if $\alpha$ is $T$-split.
\end{enumerate}
\s \label{dia-1}Tensoring the above commutative diagram  with $A/\m^n$ we obtain a commutative diagram
\[
  \xymatrix
{
 \Tor^A_1(N, A/\m^n)
    \ar@{->}[d]^{u}
\ar@{->}[r]^\gamma
 & \Tor^A_1(C, A/\m^n)
    \ar@{->}[d]
\ar@{->}[r]^\delta
& \Tor^A_1(M, A/\m^n)
    \ar@{->}[d]^{\xi}
 \\
  \Tor^A_1(N, A/\m^n)
\ar@{->}[r]^{\gamma^\prime}
 & \Tor^A_1(E, A/\m^n)
\ar@{->}[r]^{\delta^\prime}
& \Tor^A_1(M, A/\m^n)
 }
\]
Here $\xi$ is the identity map. We note that as $u \in \m^n$ we get that the multiplication map on $\Tor^A_1(N, A/\m^n)$ by $u$ is the zero map. 
Set $U = \image  \gamma$ and $U^\prime  = \image  \gamma^\prime$. Also set
$V = \image \delta$ and $V^\prime = \image \delta^\prime$.

\s \label{dia-2}   We have a commutative diagram 
  \[
  \xymatrix
{
 0
 \ar@{->}[r]
  & V
    \ar@{->}[d]^\theta
\ar@{->}[r]
 & \Tor^A_1(M, A/\m^n)
    \ar@{->}[d]^\xi
\ar@{->}[r]^\chi
& N/\m^n N
    \ar@{->}[d]^{u}
 \\
 0
 \ar@{->}[r]
  & V^\prime
\ar@{->}[r]
 & \Tor^A_1(M, A/\m^n)
\ar@{->}[r]^{\chi^\prime}
& N/\m^n N
 }
\]
where $\theta$ is induced by the maps $\xi$ and multiplication by $u$ (note $\xi$ is the identity map). Thus $\theta$ is injective and so $\ell(V) \leq \ell (V^\prime)$.

A similar argument as above yields a natural map $\eta \colon U \rt U^\prime$ which  is the zero map as the multiplication map on $\Tor^A_1(N, A/\m^n)$ by $u$ is the zero map. 

\s\label{dia-3}  We  have a commutative diagram
 
 \[
  \xymatrix
{
 0
 \ar@{->}[r]
  & U
    \ar@{->}[d]^{0}
\ar@{->}[r]
 & \Tor^A_1(C, A/\m^n)
    \ar@{->}[d]^\zeta
\ar@{->}[r]
& V
    \ar@{->}[d]^{\theta}
\ar@{->}[r]
 &0
 \\
 0
 \ar@{->}[r]
  & U^\prime
\ar@{->}[r]
 & \Tor^A_1(E, A/\m^n)
\ar@{->}[r]
& V^\prime
\ar@{->}[r]
&0
 }
\]
 Note that as $\theta$ is injective we have that $\ker \zeta \cong U$. Set $\coker \zeta = W$.
 We get
 \begin{align}
 e^T(E) - e^T(C) &= \ell(\Tor^A_1(E, A/\m^n)) - \ell(\Tor^A_1(C, A/\m^n))  \  \  \ \text{by Lemma \ref{dim1-lemm}(2)} \\
 &= \ell(W) - \ell(U) \\
 \label{eq-G}&= \ell(U^\prime) + \ell(V^\prime/V) - \ell(U).
 \end{align}
 We get the last quality by using the Snake Lemma on the commutative diagram \ref{dia-3}.
 
 We now assert that $\ell(U) \leq \ell(U^\prime)$. To see this consider the following commutative diagram
\s \label{dia-4}
 \[
  \xymatrix
{
 0
 \ar@{->}[r]
 & D
    \ar@{->}[d]^{\theta^\prime}
\ar@{->}[r]
  & \Tor^A_2(M, A/\m^n)
    \ar@{->}[d]^{\xi}
\ar@{->}[r]^\rho
 & \Tor^A_1(N, A/\m^n)
    \ar@{->}[d]^{u}
\ar@{->}[r]
& U
    \ar@{->}[d]^{0}
\ar@{->}[r]
 &0
 \\
 0
 \ar@{->}[r]
 & D^\prime
\ar@{->}[r]
  & \Tor^A_2(M, A/\m^n)
\ar@{->}[r]^{\rho^\prime}
 & \Tor^A_1(N, A/\m^n)
\ar@{->}[r]
& U^\prime
\ar@{->}[r]
&0
 }
\]
(here $\xi$ is the identity map) and $\theta^\prime  \colon D \rt D^\prime$ is 
the map induced by $\xi$ and multiplication map on $\Tor^A_1(N, A/\m^n)$ by $u$.
Clearly $\theta^\prime$ is injective. Thus $\ell(D) \leq \ell(D^\prime)$.
We also have
\begin{enumerate}
\item
$\ell(U) = \ell\left(\Tor^A_1(N, A/\m^n) \right) - \ell\left(\Tor^A_2(M, A/\m^n) \right) + \ell(D)$.
\item
$\ell(U^\prime) = \ell\left(\Tor^A_1(N, A/\m^n) \right) - \ell\left(\Tor^A_2(M, A/\m^n) \right) + \ell(D^\prime)$.
\end{enumerate}
\s \label{assert-basic}It follows that $\ell(U) \leq \ell(U^\prime)$ with equality if and only if $\theta^\prime  \colon D \rt D^\prime$ is an isomorphism.
\end{proof}

We now prove our assertions.
(i) $e^T(E) \geq e^T(C)$.  This follows from \ref{eq-G} and \ref{assert-basic}.

(ii) If $\alpha$ is $T$-split then so is $u\alpha$. As $\alpha$ is $T$-split it follows that $e^T(M) + e^T(N) - e^T(C) = 0$. As As $u\alpha$ is $T$-split it follows that $e^T(M) + e^T(N) - e^T(E) = 0$. So $e^T(E) = e^T(C)$.

Conversely if $e^T(E) = e^T(C)$ we get by \ref{eq-G} and \ref{assert-basic}
 that $\ell(V) =\ell(V^\prime)$ and $\ell(U) = \ell(U^\prime)$. We also get that
 $\theta \colon V \rt V^\prime$ and $\theta^\prime \colon D \rt D^\prime$ are isomorphisms. 
 
In \ref{dia-4} let $X = \image \rho$ and let $X^\prime = \image \rho^\prime$. Furthermore let 
 $\kappa \colon X \rt X^\prime$ be the map induced by $\theta^\prime$ and $\xi$. 
 As $\theta^\prime$ and $\xi$ are isomorphism's we get that $\kappa$ is also an isomorphism. By \ref{dia-4} we also get the following commutative diagram:
 
 \s\label{dia-5}
 \[
  \xymatrix
{
 0
 \ar@{->}[r]
  & X
    \ar@{->}[d]^{\kappa}
\ar@{->}[r]
 & \Tor^A_1(N, A/\m^n)
    \ar@{->}[d]^{u}
\ar@{->}[r]
& U
    \ar@{->}[d]^{0}
\ar@{->}[r]
 &0
 \\
 0
 \ar@{->}[r]
  & X^\prime
\ar@{->}[r]
 & \Tor^A_1(N, A/\m^n)
\ar@{->}[r]
& U^\prime
\ar@{->}[r]
&0
 }
\]
As $u \in \m^n$ the multiplication map on $\Tor^A_1(N, A/\m^n)$ by $u$ is zero.
It follows that $\kappa$ is the zero map. But $\kappa$ is also an isomorphism. So
$X = X^\prime = 0$. Thus the map $\Tor^A_1(N, A/\m^n) \rt U$ is an isomorphism.

We now consider the commutative diagram \ref{dia-2}. Let $Y = \image \chi$ and 
$Y^\prime = \image \chi^\prime$. We note that  as $\theta, \xi$ are isomorphism's the induced map $\kappa^\prime \colon Y \rt Y^\prime$ is an isomorphism.
But we also have a  commutative diagram
 \[
  \xymatrix
{
 0
 \ar@{->}[r]
  & Y
    \ar@{->}[d]^{\kappa^\prime}
\ar@{->}[r]
 & N/\m^n N
    \ar@{->}[d]^u
\ar@{->}[r]
& C/\m^n C
    \ar@{->}[d]
 \\
 0
 \ar@{->}[r]
  & Y^\prime
\ar@{->}[r]
 & N/\m^n N
\ar@{->}[r]
& E/\m^n E
 }
\]
As $u \in \m^n$ the multiplication map on $N/\m^n N$ by $u$ is the zero map. So 
$\kappa^\prime $ is the zero map. But as argued before $\kappa^\prime$ is an isomorphism. It follows that $Y = Y^\prime = 0$. So by \ref{dia-2} we get that
$V \cong \Tor^A_1(M, A/\m^n)$.
Thus we have an exact  sequence
\[
0 \rt \Tor^A_1(N, A/\m^n) \rt \Tor^A_1(C, A/\m^n) \rt \Tor^A_1(M, A/\m^n) \rt 0.
\]
By \ref{dim1-lemm} we get that $e^T(\alpha) = 0$.

We now give 
\begin{proof}[Proof of Theorem \ref{dim-T}]
We consider two cases. 

Case-1:  The residue field $k = A/\m$ is uncountable. 

We prove the result by induction  on $d = \dim A$. When $d = 1$ the result follows from \ref{MT-dim1}.
We now assume $d \geq 2$ and the results is known for \CM \ local rings with uncountable residue field and dimension $d -1$.
 Let $\alpha \colon  0 \rt N \rt C \rt M \rt 0$ and let $a \in \m$.  Then for every $n \geq 1$
  the extension $a^n\alpha$ is given by the 
push-out diagram:
 \[
  \xymatrix
{
\alpha
\colon
 0
 \ar@{->}[r]
  & N
    \ar@{->}[d]^{a^n}
\ar@{->}[r]
 & C
    \ar@{->}[d]
\ar@{->}[r]
& M
    \ar@{->}[d]^{\xi}
\ar@{->}[r]
 &0
 \\
 a^n \alpha
 \colon
 0
 \ar@{->}[r]
  & N
\ar@{->}[r]
 & E_n
\ar@{->}[r]
& M
\ar@{->}[r]
&0
 }
\]
where $\xi$ is the identity map.
As $k$ is uncountable we can choose $x \in \m$ such that for all $n \geq 1$
\[
 x \ \text{is} \ A\oplus M \oplus N \oplus C \oplus E_n \oplus \Om_A(M) \oplus \Om_A(N) \oplus \Om_A(C) \oplus \Om_A (E_n) \ 
 \text{superficial}
 \]  
 (see for instance \cite[2.2]{Pu-G})).
Set $B = A/(x)$. Set $\ov{a}$ to be image of $a$ in $B$.
By \ref{mod-sup-2} we get that 
$$e^T_A(a^n \alpha) = e^T_B((a^n \alpha) \otimes B) = e^T_B(\ov{a}^n (\alpha \otimes B)). $$
By our induction hypothesis we have $ e^T_B(\ov{a}^n (\alpha \otimes B) = 0$ for 
$n \gg 0$. Therefore we have 
$e^T_A(a^n \alpha) = 0$ for $n \gg 0$. Thus $a^n \alpha \in T_A(M,N)$ for $n \gg 0$.
It then can be easily checked that there exists $n$ such that $a^n \Ext^1_A(M, N)  \subseteq T(M,N)$ for all $n \gg 0$.
Let $\m$ be generated by $a_1,\ldots, a_r$. By our previous argument we get $n_i$ such that 
$a_i^{n_i} \Ext^1_A(M, N)  \subseteq T(M,N)$. It follows that
\[
 (a_1^{n_1},\cdots, a_r^{n_r})\Ext^1_A(M,N) \subseteq T(M,N).
\]
As $(a_1^{n_1},\cdots, a_r^{n_r})$ is $\m$-primary the result follows.

Case 2: The residue field $k$ of $A$ is either finite or countably infinite, 

We choose a local flat extension $(B,\n)$ with $\m B = \n$ such that the residue field $l = B/\n$ is uncountable.
By \ref{extend} we have that $\xi \in T_A(M,N)$ if and only if $\xi \otimes 1 \in T_B(M\otimes B, N \otimes B)$.
Let $\alpha \in \Ext^1_A(M, N)$ and let $a \in \m$. Then for all $n \geq 1$ we have
\[
 e^T_A(a^n \alpha) = e^T_B((a^n \alpha) \otimes B) = e^T_B((a\otimes 1)^n(\alpha \otimes B)).
\]
By our Case 1 it follows that  $e^T_B((a\otimes 1)^n(\alpha \otimes B)) = 0$
for $n \gg 0$. So $e^T_A(a^n \alpha) = 0$ for $n \gg 0$. Thus $a^n \alpha \in T_A(M,N)$ for $n \gg 0$.
By an argument similar to case 1 we get that $\ell( \Ext^1_A(M,N)/ T_A(M,N))$ is finite.
\end{proof}

\section{Some Properties of $L^{I}(M)$}\label{Lprop}
 Throughout this section
$(A,\m)$ is a Noetherian local ring, $M$ is a \emph{\CM }\ module of dimension
$r \geq 1$ and $I$ is \emph{an ideal of definition} for
$M$ (i.e., $\ell(M/IM)$ is finite). Set $\R = A[It]$;  the Rees Algebra of $I$. In \cite[4.2]{Pu5} we proved that 
$L^{I}(M) = \bigoplus_{n\geq 0}M/I^{n}M$ is a $\R$-module.
Note that $L^I(M)$ is not finitely generated as a $\R$-module.
In this section we recall an collect few of properties of $L^{I}(M)$ which we proved in
\cite{Pu5}.
We also prove a result on associate primes of $L^I(M)$ that we need in this paper.

 \s Set $\M = \m\oplus \R(I)_+$. Let $H^{i}(-) = H^{i}_{\M}$ denote the $i^{th}$-local cohomology functor
 \wrt \ $\M$. Recall a graded $\R$-module $N$ is said to be
*-Artinian if
every descending chain of graded submodules of $N$ terminates.
For example if $E$ is a finitely generated $\R$-module then $H^{i}(E)$ is *-Artinian for all
$i \geq 0$.

\s \label{zero-lc} In \cite[4.7]{Pu5} we proved that
\[
H^{0}(L^I(M)) = \bigoplus_{n\geq 0} \frac{\wt{I^{n+1}M}}{I^{n+1}M}.
\]
\s \label{Artin}
For $L^I(M)$ we proved that for $0 \leq i \leq  r - 1$
\begin{enumerate}[\rm (a)]
\item
$H^{i}(L^I(M))$ are  *-Artinian; see \cite[4.4]{Pu5}.
\item
$H^{i}(L^I(M))_n = 0$ for all $n \gg 0$; see \cite[1.10 ]{Pu5}.
\item
 $H^{i}(L^I(M))_n$  has finite length
for all $n \in \mathbb{Z}$; see \cite[6.4]{Pu5}.
\item
$\lambda(H^{i}(L^I(M))_n)$  coincides with a polynomial for all $n \ll 0$; see \cite[6.4]{Pu5}.
\end{enumerate}

\s \label{I-FES} The natural maps $0\rt I^nM/I^{n+1}M \rt M/I^{n+1}M \rt M/I^nM \rt 0 $ induce an exact
sequence of $R(I)$-modules
\begin{equation}
\label{dag}
0 \xar G_{I}(M) \xar L^I(M) \xrightarrow{\Pi_M} L^I(M)(-1) \xar 0.
\end{equation}
We call (\ref{dag}) \emph{the first fundamental exact sequence}.  
We use (\ref{dag}) also to relate the local cohomology of $G_I(M)$ and $L^I(M)$.

\s An easy consequence of \ref{Artin} and \ref{I-FES} is that
\[
 H^i(G_I(M)) = 0 \ \text{for} \  i = 0,1,\ldots,s \quad \text{iff} \quad H^i(L^I(M)) = 0 \ \text{for} \ i = 0, 1, \ldots,s.
\]

We need the following result in this paper.
\begin{proposition}\label{ass-L}
Assume $\depth A\oplus M > 0$.
We have
$$ \Ass_{\R(I)} G_I(M) = \Ass_{\R(I)} L^I(M).$$
\end{proposition}
\begin{proof}
Set $\R = \R(I)$, $G = G_I(M)$ and $L = L^I(M)$. 
Let $L_r$ be $\R$-submodule of $L$   defined as follows:
\[
L_r = \left<\bigoplus_{n =0}^{r} \frac{M}{I^{n+1}M} \right>.
\]
Then $L_r$ is a finitely generated $\R$-module. We note that $L_r \subseteq  L_{r+1}$ for all $r \geq 0$ and $L = \bigcup_{r \geq 0} L_r$.
So $\Ass L = \bigcup_{r \geq 0} \Ass L_r$.
Note that $L_0  = G$. So $\Ass G \subseteq \Ass L$.

For all $r \geq 0$ we have an exact sequence
\begin{equation}\label{Lr-eq}
  0 \rightarrow L_{r-1} \rightarrow L_r \xrightarrow{\rho_r} G(-r) \rightarrow 0. 
\end{equation}
For $r = 1$ we get
\[
0 \rightarrow G \rightarrow L_1 \rightarrow G(-1)\rightarrow 0.
\]
So we have $\Ass L_1 \subseteq  \Ass G$. Iterating from (\ref{Lr-eq}) we get $\Ass L_r \subseteq \Ass G$ for all $r \geq 1$.
The result follows.
\end{proof}
\section{basic lemma}
\s \label{setup-BL} Throughout this section $(A,\m)$ is a \CM \ local ring of dimension $d \geq 1$, $\R = \R(\m)$ the Rees algebra of $A$ \wrt \ $\m$. Throughout if $M$ is an $A$-module then set $G(M) = G_\m(M)$ and $L(M) = L^\m(M)$.
In this section we prove Lemma \ref{T-CM}. It is convenient to prove a slightly more general version, see \ref{basic-lemma}. 
\s For $i \geq 1$ set 
$$L_i(M) = \Tor^A_i(M, L_0(A)) = \bigoplus_{n \geq 0 } \Tor^A_i(M, A/\m^{n+1}). $$
We assert that $L_i(M)$ is a finitely generated $\R$-module for $i \geq 1$. It is sufficient to prove it for $i = 1$. We tensor the exact sequence 
$0 \rt \R \rt A[t] \rt L_0(A)(-1) \rt 0$ with $M$ to obtain a sequence of $\R$-modules
\[
0 \rt L_1(M)(-1) \rt \R \otimes_A M \rt M[t] \rt  L_0(M)(-1) \rt 0.
\] 
Thus $ L_1(M)(-1)$ is a $\R$-submodule of $\R \otimes_A M$. The latter module is a finitely generated $\R$-module. It follows that $L_1(M)$ is a finitely generated $\R$-module. 

Next we show:
\begin{lemma}\label{basic-lemma}
(with assumptions as in \ref{setup-BL}.) Further assume that $M, N, E$ are
MCM $A$-modules and we have a $T$-split exact sequence $s \colon 0 \rt N \rt E \rt M \rt 0$. Assume that $G(N)$ is \CM. Then we have short exact sequence
\[
0 \rt L(N) \rt L(E) \rt L(M) \rt 0;
\]
and hence a short-exact sequence 
\[
0 \rt G(N) \rt G(E) \rt G(M) \rt 0.
\]
Furthermore $e_i(E) = e_i(N) +  e_i(M)$ for $i = 0,1,\ldots, d$. 
\end{lemma}
\begin{proof}
It is clear that $e_0(E) = e_0(N) +  e_0(M)$. We prove rest of the assertion by induction on dimension $d \geq 1$.

We first consider the case $d = 1$. As $s$ is $T$-split we have that
\[
0 \rt L_1(N)_n \rt L_1(E)_n \rt L_1(M)_n \rt 0
\]
is exact for $n \gg 0$, see \ref{filmy}. It follows we have an exact sequence of 
$\R$-modules
\[
0 \rt K \rt L(N) \rt L(E) \rt L(N) \rt 0;
\]
where $\lambda(K) < \infty$. As $H^0_\M(L(N)) = 0$ (since $H^0_\M(G(N)) = 0$) we get that $K = 0$.
Thus we have an exact sequence $0 \rt L(N) \rt L(E) \rt L(M) \rt 0$. In particular we have $e_1(E) = e_1(M) + e_1(N)$.

Now assume that $d \geq 2$ and the result has been proved for all \CM \ local rings and all $T$-split sequences satisfying our hypothesis.
By \ref{mod-sup-2} it follows that if $x \in \m$ is sufficiently  general then
$s \otimes A/(x)$ is $T$-split. Set $\ov{(-)} = (-)\otimes (A/(x))$. So we get that $e_i( \ov{E}) = e_i(\ov{M}) + e_i(\ov{N})$ for $i = 0,\ldots, d-1$.
By our choice of $x$ we get that $e_i(E) = e_i(N) +  e_i(M)$ for $i = 0,1,\ldots, d-1$. Tensoring $s$ with $L(A)$ we get sequence of $\R$-modules
\[
 L_1(E) \rt L_1(M) \xrightarrow{\delta} L(N) \rt L(E) \rt L(M) \rt 0.
\]
Set $K$ = $\image \delta$. As $L_1(M)$ is finitely generated as a $\R$-module we get that $K$ is finitely generated as an $\R$-module. We have exact sequence 
\[
0 \rt K \rt L(N) \rt L(E) \rt L(M) \rt 0.
\]
As $e_i(E) = e_i(N) +  e_i(M)$ for $i = 0,1,\ldots, d-1$ we get $\dim K \leq 1$ 
as a $\R$-module. Suppose if possible $K \neq 0$. Let $P \in \Min K$. Then $P \in \Ass L(N) = \Ass_\R G(N)$ (by \ref{ass-L}). But $G(N)$ is \CM \ $\R$-module of dimension $d$. So $\height P  = 1$. This implies that $\dim K \geq  d$. But $d \geq 2$ and we have shown that $\dim K \leq 1$. It follows that $K = 0$. Thus we have an exact sequence $0 \rt L(N) \rt L(E) \rt L(M) \rt 0$. In particular we have $e_i(E) = e_i(M) + e_i(N)$ for $i = 0,\ldots, d$.

To prove the result regarding associated graded modules, note that we have the following commutative diagram
\[
  \xymatrix
{
s
\colon
 0
 \ar@{->}[r]
  & L(N)
    \ar@{->}[d]^{\Pi_N}
\ar@{->}[r]
 & L(E)
    \ar@{->}[d]^{\Pi_E}
\ar@{->}[r]
& L(M)
    \ar@{->}[d]^{\Pi_N}
\ar@{->}[r]
 &0
 \\
 s(-1)
 \colon
 0
 \ar@{->}[r]
  & L(N)(-1)
\ar@{->}[r]
 & L(E)(-1)
\ar@{->}[r]
& L(M)(-1)
\ar@{->}[r]
&0
 }
\]
The result follows by applying (\ref{dag}) and the Snake  Lemma.
\end{proof}
As an immediate corollary we get
\begin{corollary}\label{bella}
(with hypotheses as in \ref{basic-lemma}.) Further assume that $G(M)$ is also \CM.
Then $G(E)$ is \CM.
\end{corollary}
\section{Proof of Theorem \ref{main-sect-intro}}
In this section we give a proof of Theorem \ref{main-sect-intro}. We restate it here for the reader's convenience.
We also indicate two other results whose proofs are parallel to our main result.
\begin{theorem}\label{main-sect}
Let $(A,\m)$ be a  Henselian \CM \ local ring of dimension $d \geq 1$. Suppose $M, N$ are MCM  modules with $G(M), G(N)$ \CM. If there exists only finitely many non-isomorphic
MCM $A$-modules $D$ with $G(D)$ \CM \ and $e(D) = e(M) + e(N)$; then $T_A(M, N)$ has finite length (in particular $\Ext^1_A(M, N)$ has finite length). If $h$  is the number of such isomorphism classes then $\m^h$
annihilates $T_A(M, N)$.
\end{theorem}
\begin{remark}
Before we prove Theorem  \ref{main-sect} note that it is formally similar to the statement of a result by Huneke and  Leuschke \cite[Theorem 1]{HL}. The proof is similar too except in one detail which we describe.
\end{remark}
\begin{proof}
Let $\chi \in T_A(M, N)$.  Let $r_1, \ldots, r_h \in \m$. We prove $r_1\cdots r_h \chi = 0$.
Let 
$$ \chi  \colon  0 \rt N \rt K \rt M \rt 0.$$
Consider 
\[
r_1\cdots r_n\chi   \colon  0 \rt N \rt K_n \rt M \rt 0.
\]
where $n$ runs through all positive integers and each $r_i \in \m$. 
First note that $K_n$ is a MCM $A$-module for all $n \geq 1$.
Further note that
as $T_A(M, N)$ is an $A$-submodule of $\Ext^1_A(M, N)$ we get $r_1\cdots r_n\chi$ is $T$-split. In particular by \ref{bella} we get that $G(K_n)$ is \CM \ for all $n \geq 1$.  Furthermore notice $e(K_n) = e(M) + e(N)$. By our assumption
there must be repetitions among the $K_n$. The rest of the proof is similar to 
\cite[Theorem 1]{HL}.
\end{proof}
The following two results have proofs similar to Theorem \ref{main-sect}.
\begin{theorem}\label{main-Ulrich-sect}
Let $(A,\m)$ be a  Henselian \CM \ local ring of dimension $d \geq 1$. Suppose $M, N$ are   Ulrich modules.
If there exists only finitely many non-isomorphic
Ulrich $A$-modules $D$ with  $e(D) = e(M) + e(N)$; then $T_A(M, N)$ has finite length (in particular $\Ext^1_A(M, N)$ has finite length). If $h$  is the number of such isomorphism classes then $\m^h$
annihilates $T_A(M, N)$.
\end{theorem}
\begin{proof}
We only note that if $\chi \colon  0\rt  N \rt E \rt M \rt 0$ is in $T_A(M,N)$ and 
both $M$ and $N$ Ulrich then $E$ is MCM and $G(E)$ is \CM.  By \ref{basic-lemma} $e_1(E) = e_1(M) + e_1(N) = 0$.
So $E$ is Ulrich. The rest of the proof is similar to  proof of Theorem \ref{main-sect}.
\end{proof}
Our next result is
\begin{theorem}\label{main-min-mult-sect}
Let $(A,\m)$ be a  Henselian \CM \ local ring of dimension $d \geq 1$. Suppose $M, N$ are MCM modules with minimal multiplicity. If there exists only finitely many non-isomorphic MCM  $A$-modules $D$ having minimal multiplicity and $e(D) = e(M) + e(N)$; then $T_A(M, N)$ has finite length (in particular $\Ext^1_A(M, N)$ has finite length). If $h$  is the number of such isomorphism classes then $\m^h$
annihilates $T_A(M, N)$.
\end{theorem}
\begin{proof}
We first note that as $M, N$ have minimal multiplicity both $G(M), G(N)$ are \CM.
Also note that if $\chi \colon  0\rt  N \rt E \rt M \rt 0$ is in $T_A(M,N)$ and 
both $M$ and $N$ having minimal multiplicity. then $E$ is MCM  and $G(E)$ is \CM. Note $e_2(E) = e_2(M) + e_2(N) = 0$ 
see \ref{basic-lemma}.
So $E$ has minimal multiplicity. The rest of the proof is similar to  proof of Theorem \ref{main-sect}.
\end{proof}
\section{Strict complete intersections}
In this section we prove the following result.
\begin{theorem}\label{sci}
Let $(Q,\n)$ be a Henselian regular local ring and let \\ $A = Q/(f_1,\ldots, f_c)$ be a strict complete intersection.
Let $f_1 = g^ih$ with $g$ irreducible, ($h$ is possibly a unit if $i \geq 2$ and is a non-unit if $i = 1$) and
$g$ does not divide $h$. If $i \geq 2$ assume $\dim A \geq 1$.
If $i = 1$ assume $\dim A \geq 2$. Then there exists $\{ E_n \}_{n \geq 1}$ indecomposable MCM $A$-modules with bounded multiplicity and having
$G(E_n)$ \CM \ for all $n \geq 1$. 
\end{theorem}
The proof follows by first analyzing $Q/(f_1)$. The following result is easy to prove.
\begin{proposition}\label{hypersurface-prop}
 Let $(Q,\n)$ be a Henselian regular local ring and let $A = Q/(f)$ with $f \in n^2$.
Let $f = g^ih$ where $g$ is irreducible ($h$ is possibly a unit if $i \geq 2$ and is a non-unit if $i =1$) and $g$ does not divide $h$.
Then
\begin{enumerate}[\rm (1)]
 \item $A/(g) = Q/(g)$ and $A/(g^{i-1}h) = Q/(g^{i-1}h)$.
 \item $gA$ is a prime ideal in $A$ of height zero.
 \item $A/(g)$ and $A/(g^{i-1}h)$ are MCM $A$-modules with \CM \ associated graded modules.
 \item $gA \cong Q/(g^{i-1}h)$ and $(g^{i-1}h) A \cong Q/(g)$.
\item
The following is a minimal periodic free resolution of $A/(g)$.
\[
\cdots \rt A\xrightarrow{g^{i-1}h} A \xrightarrow{g} A \rt 0.  
\]
 \end{enumerate}
\end{proposition}
We now give an estimate of dimension of a certain Ext module.
\begin{proposition}\label{hyp-ext} (with hypotheses as in \ref{hypersurface-prop}
 Let $D = \Ext^1_A(A/(g), A/(g^{i-1}h))$. If $i \geq 2$ then $\dim D = \dim A$. If $i = 1$ then $\dim D \geq \dim A - 1$.
\end{proposition}
\begin{proof}
 We first consider the case $i \geq  2$. Consider the short exact sequence
 $$ s \colon 0 \rt A/(g^{i-1}h) \rt A \rt A/(g) \rt 0. $$
 Let $P = gA$ a prime ideal of height zero in $A$. If $s_P = 0$ then note that
 $\kappa(P) = (A/(g))_P$ the residue field of $P$ is a free $A_P$-module. This implies that $A_P$ is regular. However note that
 $A_P = Q_{(g)}/(g^i)$ has nilpotent elements, a contradiction. So $s_P \neq 0$. The result follows.
 
 We now consider the case $i = 1$. Let $\mathfrak{\beta}$ be a prime ideal of height two in $Q$ minimal over $g, h$. Then
 $\mathfrak{q} = \beta/(f)$ is a height one prime ideal in $A$. Consider the short exact sequence
 $$ s \colon 0 \rt A/(h) \rt A \rt A/(g) \rt 0. $$
 Note $(A/(h))_\q$ and $(A/(g))_\q$ are non-zero. Therefore $s_\q \neq 0$. The result follows.
\end{proof}
We now give a proof of Theorem \ref{sci} in the case of hypersurfaces.
\begin{proof}
 We note that both $G(A/(g))$ and $G(A/(g^{i-1}h))$ are \CM. By Proposition \ref{hyp-ext} under our assumptions we have
 $\dim \Ext^1_A(A/(g), A/(g^{i-1}h)) \geq 1$. By Theorem \ref{main-sect} 
 the result follows.
\end{proof}
To prove Theorem \ref{sci} in general we need the next two results which are certainly known to experts. However we give proofs as we cannot find a 
reference. Before stating the result we make the convention that dimension of the zero module is $-1$.
\begin{lemma}\label{eclair-1}
 Let $(A,\m)$ be a \CM \ local ring and let $E$ be a finitely generated $A$-module with $\dim E = r \geq 1$. Let $x \in \m$.
 Then $\dim E/xE \geq r -1$.
\end{lemma}
\begin{proof}[Sketch of a proof]
 By Nakayama's Lemma  $E/xE \neq 0$. So we have nothing to show when $r = 1$. So assume $\dim E \geq 2$. Let 
 $P$ be a prime ideal in $A$ with $P \in Supp(E)$ and $\dim A/P = r$. If $x\in P$ then choose $\q = P$. Otherwise choose $\q$ minimal
 over $(P,x)$. Then by Nakayama's Lemma $(E/xE)_\q \neq 0$. The result follows.
\end{proof}
An easy consequence of the above result is:
\begin{corollary}\label{eclair-2}
 Let $(A,\m)$ be a \CM \ local ring and let $M, N$ be MCM $A$-modules with $\dim \Ext^1_A(M, N) = r \geq 1$. Let 
 $x \in \m$ be a non-zero divisor in $A$. Set $B = A/(x)$, $\ov{M} = M/xM$ and $\ov{N} = N/xN$. Then
 $\dim \Ext_B^1(\ov{M},\ov{N}) \geq r -1$.
\end{corollary}
\begin{proof}[Sketch of a proof]
Note that $x$ is $M \oplus N$-regular. The exact sequence $ 0 \rt N \xrightarrow{x} N \rt \ov{N} \rt 0$ induces
an exact sequence
\[
 \Ext^1_A(M,N) \xrightarrow{x} \Ext^1_A(M, N) \rt \Ext^1_A(M, \ov{N}) = \Ext_B^1(\ov{M}, \ov{N})
\]
The result now follows from Lemma \ref{eclair-1}.
\end{proof}
We now give 
\begin{proof}[Proof of Theorem \ref{sci}]
We have already proved the result for $c = 1$. Assume $c \geq 2$. Set $R = Q/(f_1)$. Then
$M = R/(g)$ and $N = R/(g^{i-1}h)$ are MCM $A$-modules with \CM \ associated graded modules. Note that $f_2^*,\cdots, f_c^*$ are
$G(R)$-regular and so $G(M), G(N)$ regular sequence. Set $\ov{M} = M/(f_2,\cdots, f_c)$ and $\ov{N} = N/(f_2,\cdots, f_c)$. Then
$\ov{M}$ and $\ov{N}$ are maximal \CM \ $A$-modules with \CM \ associated graded modules. Using \ref{hyp-ext} and \ref{eclair-2}
we get that $\dim \Ext_A^1(\ov{M}, \ov{N}) \geq 1$. The result follows from \ref{main-sect}.
\end{proof}

%%%%%%%%exit%%%%%%%%%%%%%%%%%%%%%%%%%%%%%%%%%%%%%%%%%%%%%%%%%%%%%%%%%%%%%%%%%%%%%%%%%%%%%%%%%%%%%%%%%%%%%%%%%%%%%%%%%
\section{Small Dimensions}
In this section we show that if $(A,\m)$ is \CM \ of dimension two with $G(A)$ \CM \  then there exists a 
MCM $A$-module
$M$ with $G(M)$ \CM. If $A$ is not an isolated singularity then we show that there exists an MCM module $M$ 
with $G(M)$ \CM \ and 
$\dim \Ext^A_1(M, M) \geq 1.$  As a consequence we prove weak Brauer-Thrall II for non-isolated singularities of dimensions 
$1,2$.
We also prove a few preliminary results on Ulrich modules over a one dimensional \CM \ local ring.

\s  
The following result showing existence of Ulrich modules in one-dimensional \CM \ local rings is well-known. We give a proof
due to lack of a reference.
\begin{proposition}\label{Ulrich-d1}
 Let $(A,\m)$ be a one-dimensional \CM \ local ring and let $E$ be a MCM $A$-module. Then for all $n \gg 0$ the modules $\m^n E$
 are Ulrich $A$-modules.
\end{proposition}
\begin{proof}
 As $\dim E = 1$ there exists $n_0$ with $e = e(E) = \mu(\m^n E)$ for all $n \geq n_0$. Fix $n \geq n_0$. Set $M = \m^n E$. Then 
 $M$ is a MCM $A$-module. Furthermore $e(M) = e$. By construction $\mu(M) = e$. Thus $M$ is an Ulrich $A$-module.
\end{proof}
The following result is required in section \ref{sect-Ulrich}.
\begin{proposition}\label{dim1-NIso-Ulrich}
 Let $(A,\m)$ be a one-dimensional \CM \ local ring \\ which is not an isolated singularity (equivalently $A$ is not reduced). Then there
 exists an Ulrich $A$-module $M$ with $\dim \Ext^1_A(M, M) = 1$
\end{proposition}
\begin{proof}
 Let $P$ be a minimal prime of $A$ with $A_P$ not a regular local ring (equivalently $A_P$ is not a field). Let $M$ be an Ulrich
 $A/P$-module. Then $M$ is also an Ulrich $A$-module. Notice $M_P = \kappa(P)^r$ for some $r \geq 1$ (here $\kappa(P)$ is the 
 residue field of $A_P$). 
 
 We have 
 \[
  \Ext^1_A(M, M)_P = \Ext^1_{A_P}(M_P, M_P) = \Ext^1_{A_P}(\kappa(P)^r, \kappa(P)^r) \neq 0.
 \]
The last assertion holds since $A_P$ is not a regular ring. Thus $\dim_A \Ext^1_A(M, M) \neq 0$. The result follows.
\end{proof}

We now show one of the main results of this section.
\begin{theorem}\label{dim2}
 Let $(A,\m)$ be a two dimensional \CM \ local ring with $G(A)$ \CM. Then there exists an MCM $A$-module $M$ with $G(M)$ \CM.
 \end{theorem}
 \begin{proof}
  Let $B$ be a one-dimensional \CM \ quotient of $A$ and let $E$ be an Ulrich $B$-module. Let $M = \Syz^A_1(E)$. We claim that 
  $G(M)$ is \CM. Without loss of any generality we may assume that the residue field $k$ of $A$ is infinite. Let $x$ be $A \oplus M \oplus E$-superficial.
  Then $E/xE = k^r$ for some $r \geq 1$. Then $M/x M = \n^r$ where $\n$ is the maximal ideal of $C = A/(x)$. Note that $G(C)$ is \CM \ and so
  $G(\n)$ is \CM. Thus $G(M/xM)$ is \CM. By Sally descent we have $G(M)$ is \CM.
 \end{proof}
We will also need the following result.
\begin{proposition}\label{dim2-niso}
 Let $(A,\m)$ be a two dimensional \CM \ local ring with $G(A)$ \CM. Assume $A$ is not an isolated singularity. Then 
 there exists
 an MCM $A$-module $M$ with $G(M)$ \CM \ and $\dim \Ext^A_1(M, M) \geq 1$.
\end{proposition}
\begin{proof}
 Let $P$ be a height one prime with $A_P$ not regular. Let $E$ be an Ulrich $A/P$-module, By proof of Theorem \ref{dim2} we
 get that $M = \Syz^A_1(E)$ has $G(M)$ \CM.  
 
 Claim: $\dim \Ext_A^1(M,M) \geq 1$. \\
 It suffices to prove that $\Ext^1_A(M, M)_P \neq 0$. Notice $E_P = \kappa(P)^r$ for some $r \geq 1$ (here $\kappa(P)$ is the 
 residue field of $A_P$). It follows that $M_P \cong \n^r \oplus A_P^s$ where $\n$ is the maximal ideal of $A_P$. By Lemma \ref{non-reg} we have that
 $\Ext^{1}_{A_P}(\n, \n) \neq 0$. This proves the result.
\end{proof}
We need the following result in the proof of Proposition \ref{dim2-niso}.
I believe that this already known to the experts. We give a proof due to lack of a suitable reference. 
\begin{lemma}\label{non-reg}
 Let $(S,\n)$ be a one dimensional non-regular local ring. Then \\ $\Ext^1_S(\n,\n) \neq 0$.
\end{lemma}
\begin{proof}
We first assert that $\injdim_S \n = \infty$. Suppose if possible $\injdim \n < \infty$. Let $x \in \n \setminus \n^2$ be $S$-regular. So it is also $\n$-regular. Set $R = S/(x)$ and $k = $ residue field of $R$. As 
$\injdim_S \n < \infty$ we get $\injdim_R \n/x\n < \infty$. We have a split exact sequence of $R$-modules:
\[
0 \rt k \rt \n/x\n \rt \n/(x) \rt 0.
\]
It follows that $\injdim_R k < \infty$. So $R$ is regular and as $x \in \n \setminus \n^2$ we get that $S$ is also regular, a contradiction.

Suppose if possible $\Ext^1_S(\n,\n) = 0$. By applying the functor $\Hom_S(-, \n)$ to the exact 
sequence $ 0 \rt \n \rt S \rt k \rt 0$ we get $\Ext^2_S(k,\n) = 0$.
This is a contradiction as $\dim S = 1$ and $\injdim \n = \infty$,  see \cite{BH} exercise problem 3.5.12(b).
\end{proof}
An easy consequence of our previous results is weak Brauer-Thrall II for associated graded modules in dimensions one and two.
\begin{theorem}\label{BT-1-2}
 Let $(A,\m)$ be a \CM \ local ring of dimension one or two. Assume $A$ is not an isolated singularity and that $G(A)$ is \CM.
 Then weak Brauer-Thrall II holds for associated graded modules of $A$.
\end{theorem}
\begin{proof}
This follows from \ref{dim1-NIso-Ulrich}, \ref{dim2-niso} and \ref{main-sect}.
\end{proof}

%%%%%%%%%%%%%%%%%%%%%%%%%%%%%exit%%%%%%%%%%%%%%%%%%%%%%%%%%%%%%%%%%%%%%%%%%%%%%%%%%%%%

 \section{weak Brauer-Thrall II for Ulrich modules and for relative complete intersections}\label{sect-Ulrich}
In this section we discuss our results regarding weak Brauer-Thrall II for Ulrich modules and for relative complete intersections. 
In dimension one we have the following:
\begin{proposition}\label{BT-2-U-1}
 Let $(A,\m)$ be a  one-dimensional \CM \ local ring with $G(A)$ \ \CM. Assume $A$ is not an isolated singularity. Then
 $A$ satisfies weak Brauer-Thrall II for Ulrich modules.
\end{proposition}
\begin{proof}
 This follows  from \ref{dim1-NIso-Ulrich} and \ref{main-Ulrich-sect}.
\end{proof}
The main result of this section is
\begin{proposition}
 Let $(A,\m)$ be a \CM \ local ring with $G(A)$ \ \CM. Assume $A$ has an Ulrich module $U$. Let $r \geq 1$ and let $B$ be either 
 $A[X_1,\ldots, X_r]_{(\m, X_1,\ldots, X_r)}$ or $A[[X_1,\ldots, X_r]]$. Then
 $B$ satisfies weak Brauer-Thrall-II for Ulrich modules.
\end{proposition}
\begin{proof}
 Note that $B$ is a flat extension of $A$ with an $r$-dimensional fiber. Also note that $M = U\otimes_A B$ is an Ulrich $B$-module
 with $\dim \Ext_B^1(M, M)  \geq r$ (see \cite[Theorem A.11(b)]{BH}). The result now follows from \ref{main-Ulrich-sect}.
\end{proof}

Finally we prove weak Brauer-Thrall II for relative complete intersections.

\begin{theorem}\label{rci}
  Let $(R,\n)$ be a \CM \ local ring having a non-free MCM module $E$ with $G(E)$ \CM.  Also assume $G(R)$ is \CM. Let $r \geq 1$ and let
$B = A[[X_1,\ldots, X_r]]$ or $B = R[X_1,\ldots, X_r]_{(\n, X_1,\ldots, X_r)}$. Note $G(B)$ is \CM.  Let $0 \leq l \leq r -1$ and let
$g_1,\ldots, g_l$ be such that $g_1^*, \ldots, g_l^*$ is $G(B)$ regular.
Set $A = B/(g_1,\ldots, g_l)$. Then 
A satisfies weak Brauer-Thrall II for associated graded modules.
\end{theorem}
\begin{proof}
 Set $N = E\otimes_RB$. Notice $N$ is a non-free maximal \CM \ $B$ module with $G(N)$ \CM.
 $B$ is a flat extension of $A$ with an $r$-dimensional fiber. 
 We have $\dim \Ext_B^1(N, N)  \geq r$ (see \cite[Theorem A.11(b)]{BH}).

 As $G(N)$ is a MCM $G(B)$-module we have that $g_1^*, \ldots, g_l^*$ is $G(N)$ regular.
Set $M = N/(g_1,\ldots, g_l)N$. We note that $G(M)$ is \CM. By \ref{eclair-2} we get that
$\dim \Ext^1_A(M, M) \geq 1$. The result follows from Theorem \ref{main-sect}.
\end{proof}

\end{document}